\documentclass[11pt]{article}
\usepackage[utf8]{inputenc}
\usepackage{amsmath}
\usepackage[hmargin=1in,vmargin=1in]{geometry}
\usepackage{graphics}
\usepackage[toc,page]{appendix}
\usepackage{amsfonts}
\usepackage{MnSymbol}
\usepackage{wasysym}
\usepackage{mathtools}
\usepackage{enumerate}
\usepackage{amsthm}
\usepackage{mathtools}
\usepackage{graphicx}
\usepackage{float}
\usepackage{chngcntr}
\usepackage{thmtools, thm-restate}
\theoremstyle{plain}
\newtheorem{thm}{Theorem}[section]
\newtheorem{lemma}[thm]{Lemma}

\newtheorem{prop}[thm]{Proposition}

\theoremstyle{definition}

\numberwithin{equation}{section}
\counterwithin{figure}{section}
\title{Remark on a Simple Proof of the Mean Value of $K_2(\mathcal{O})$ in Function Fields}
\author{J. MacMillan}
\date{}
\newcommand\blfootnote[1]{%
  \begingroup
  \renewcommand\thefootnote{}\footnote{#1}%
  \addtocounter{footnote}{-1}%
  \endgroup
}
\begin{document}
\maketitle
\blfootnote{Date: October 10, 2019  }
\par\noindent
ABSTRACT: Let $\mathbb{F}_q$ denote a finite field of odd cardinality $q$, $\mathbb{A}=\mathbb{F}_q[T]$ the polynomial ring over $\mathbb{F}_q$ and $k=\mathbb{F}_q(T)$ the rational function field over $\mathbb{F}_q$. In this paper, we compute the average value of the size of the group $K_2(\mathcal{O}_{\gamma D})$, where $\mathcal{O}_{\gamma D}$ denotes the integral closure of $\mathbb{A}$ in $k(\sqrt{\gamma D})$, $D$ is a monic, square-free polynomial of even degree and $\gamma$ is a fixed generator of $\mathbb{F}_q^*$.
\section{Introduction}
Let $k=\mathbb{F}_q(T)$ denote the rational function field over a finite field $\mathbb{F}_q$ and let $\mathbb{A}=\mathbb{F}_q[T]$ denote the polynomial ring of the finite field $\mathbb{F}_q$, where $q$ is assumed to be odd and greater than 3. For $f\in\mathbb{A}$, the norm of $f$, $|f|$ is defined to be $q^{\text{deg}(f)}$ if $f$ is non-zero and 0 otherwise. For $\Re(s)>1$, the zeta function associated with $\mathbb{A}$ is defined by
\begin{equation}
    \zeta_{\mathbb{A}}(s):=\sum_{f\in\mathbb{A}^+}\frac{1}{|f|^s}=\prod_P(1-|P|^{-s})^{-1},
\end{equation}
where the letter $P$ denotes a monic, irreducible polynomial in $\mathbb{A}$, $\mathbb{A}^+$  denotes the set of all monic polynomials in $\mathbb{A}$ and $\mathbb{A}^+_n$ denotes the set of monic polynomials of degree $n$ in $\mathbb{A}$. Since $\#\mathbb{A}^+_n=q^n$, then $\zeta_{\mathbb{A}}(s)=(1-q^{1-s})^{-1}$. For $D\in\mathbb{A}$ square-free, the quadratic Dirichlet character $\chi_D(f)$ is equal to the Kronecker symbol $\left(\frac{D}{f}\right)$. Therefore we can define the Dirichlet L-function corresponding to the Dirichlet character $\chi_D$ as
\begin{equation}\label{eq:1.2}
    L(s,\chi_D):=\sum_{f\in\mathbb{A}^+}\frac{\chi_D(f)}{|f|^s}.
\end{equation}
We define $\mathcal{O}_D$ to be the integral closure of $\mathbb{A}$ in the quadratic function field $K_D=k(\sqrt{D})$. The zeta function of the ring $\mathcal{O}_D$ is defined as
\begin{equation}
    \zeta_{\mathcal{O}_D}(s)=\sum_{\mathfrak{a}}N\mathfrak{a}^{-s},
\end{equation}
where $\mathfrak{a}$ runs through all non-zero ideals of $\mathbb{A}[\sqrt{D}]$ and $N\mathfrak{a}$ denotes the norm of $\mathfrak{a}$. From \cite{Rosen2002}, Proposition 17.7, we have the relation
\begin{equation}
    \zeta_{\mathcal{O}_D}(s)=\zeta_{\mathbb{A}}(s)L(s,\chi_D).
\end{equation}
\par\noindent
Let $F=\mathbb{F}_q$ and let $K \backslash F$ be a function field in one variable defined over the finite constant field $\mathbb{F}_q$.  The primes in $K$ are denoted by $v$ and $\mathcal{O}_v$ is the valuation ring at $v$. Let $\mathcal{P}_v$ denote the maximal ideal of $\mathcal{O}_v$ and $\bar{F}_v$ by the residue class of $v$. The tame symbol $(*,*)_v$ is a mapping from $K^*\times K^*$ to $\bar{F}_v$ which is defined as
\begin{equation}
    (a,b)_v=(-1)^{v(a)v(b)}\frac{a^{v(b)}}{b^{v(a)}}\text{ mod }\mathcal{P}_v.
\end{equation}
Let $a\in K^*$ be such that $a\neq 0,1$. The group $K_2(K)$ is defined to be $K^*\bigotimes K^*$ modulo the subgroup generated by the elements $a\bigotimes (1-a)$. Moore, (see \cite{Tate1971}), proved that the following sequence is exact
\begin{equation}
    (0)\rightarrow\text{ker}(\lambda)\rightarrow K_2(K)\xrightarrow{\lambda}\bigoplus_v\bar{F}^*_v\xrightarrow{\mu}F^*\rightarrow (0),
\end{equation}
where $\lambda:K_2(K)\rightarrow\bigoplus_v\bar{F}^*_v$ is the sum of the tame symbol maps and $\mu:\bigoplus_v\bar{F}^*_v\rightarrow F^*$ is the map given by $\mu(\dotsc,a_v,\dotsc)=\prod_va_v^{m_v/m}$ where $m_v=N\mathcal{P}_v-1$ and $m=|F^*|=q-1$. In \cite{Tate1971}, Tate gave a proof of the Birch-Tate conjecture concerning the size of ker$(\lambda)$. He proved that 
\begin{equation}
    |\text{ker}(\lambda)|=(q-1)(q^2-1)\zeta_K(-1)
\end{equation}
where $\zeta_K(s)=\prod_v(1-N\mathcal{P}_v^{-s})^{-1}$, the product being over all primes $v$ over the function field $K$.\\
\par\noindent
For $m\in\mathbb{A}, m$ square-free, Rosen, \cite{Rosen1995AverageFields}, was was able to relate the number $L(2,\chi_m)$ to the size of the group $K_2(\mathcal{O}_m)$ using the notation above. 
\begin{prop}
Let $K_m=k(\sqrt{m})$, where $m$ is a square-free polynomial of degree $M$  in $\mathbb{A}$ . Then
\begin{enumerate}[a)]
    \item If M is odd then 
    \begin{equation}
        K_2(\mathcal{O}_m)=q^{\frac{3M}{2}}q^{-\frac{3}{2}}L(2,\chi_m).
    \end{equation}
    \item If M is even and the and the leading coefficient of $m$ is not a square, then
    \begin{equation}
        K_2(\mathcal{O}_m)=q^{\frac{3M}{2}}(q+1)q^{-1}(q^2+1)^{-1}L(2,\chi_m).
    \end{equation}
\end{enumerate}
\end{prop}
\par\noindent
Using the Proposition, Rosen proved the following result. 
\begin{thm}
Let $m$ be a square-free polynomial of degree $M$ in $\mathbb{A}$ and $\epsilon>0$. Then 
\begin{enumerate}[a)]
    \item for $M$ odd, we have
    \begin{equation}
        (q-1)^{-1}(q^M-q^{M-1})^{-1}\sum_{\substack{m\in\mathbb{A}\\m\text{square-free}}}|K_2(\mathcal{O}_m)|=\zeta_{\mathbb{A}}(2)\zeta_{\mathbb{A}}(4)c(2)q^{\frac{3M}{2}}q^{-\frac{3}{2}}+O(q^{M(1+\epsilon)}).
    \end{equation}
    \item for $M$ even and the leading coefficient of $m$ is not a square, then 
    \begin{equation}
        2(q-1)^{-1}(q^M-q^{M-1})\sum_{\substack{m\in\mathbb{A}\\m\text{ square-free}}}|K_2(\mathcal{O}_m)|=q^{\frac{3M}{2}}\zeta_{\mathbb{A}}(2)\zeta_{\mathbb{A}}(4)(q+1)q^{-1}(q^2+1)^{-1}c(2)+O(q^{M(1+\epsilon)}),
    \end{equation}
\end{enumerate}
where 
\begin{equation*}
    c(2)=\prod_P\left(1-|P|^{-2}-|P|^{-5}+|P|^{-6}\right)
\end{equation*}
\end{thm}
\par\noindent
Let $\mathcal{H}_n$ denote the set of all monic, square-free polynomials of degree $n$ in $\mathbb{A}$. In \cite{SimpleLaValeurMoyenneDe}, Andrade computed the size of the group $K_2(\mathcal{O}_D)$ for $D\in\mathcal{H}_{2g+1}$. In particular he proved the following.
\begin{thm}
Let $D\in\mathcal{H}_{2g+1}$ and $\epsilon>0$. Then
\begin{equation}
    \frac{1}{\#\mathcal{H}_{2g+1}}\sum_{D\in\mathcal{H}_{2g+1}}\#K_2(\mathcal{O}_D)=q^{\frac{3}{2}(2g+1)}q^{-\frac{3}{2}}\zeta_{\mathbb{A}}(4)P(4)+O(q^{(2g+1)(1+\epsilon)}),
\end{equation}
where 
\begin{equation}\label{eq:1.13}
    P(s)=\prod_P\left(1-\frac{1}{|P|^{s}(|P|+1)}\right).
\end{equation}
\end{thm}
\par\noindent
Using the methods used by Andrade, the aim of this paper is to calculate the size of the group $K_2(\mathcal{O}_{\gamma D})$ for $D\in\mathcal{H}_{2g+2}$ and $\gamma$ a fixed generator of $\mathbb{F}_q^*$, which is the following Theorem.
\begin{thm}
Let $D\in\mathcal{H}_{2g+2}, \gamma$ a fixed generator of $\mathbb{F}_q^*$ and $\epsilon>0$. Then
\begin{equation}
    \frac{1}{\#\mathcal{H}_{2g+2}}\sum_{D\in\mathcal{H}_{2g+2}}\#K_2(\mathcal{O}_{\gamma D})=q^{\frac{3}{2}(2g+1)}\frac{\zeta_{\mathbb{A}}(2)\zeta_{\mathbb{A}}(4)P(4)}{\zeta_{\mathbb{A}}(5)}+O(q^{(2g+2)(1+\epsilon)}),
\end{equation}
where $P(s)$ is given by equation (\ref{eq:1.13}).
\end{thm}
\section{Preliminaries}
Before stating results necessary to prove Theorem 1.4, we state some facts about the quadratic Dirichlet L-function $L(s,\chi_D)$, where $D\in\mathbb{A}, D$ square-free. From \cite{Rosen2002}, Proposition 4.3, we have that $L(s,\chi_D)$ is a polynomial in $q^{-s}$ of degree at most deg$(D)-1$. Using the change of variable $u=q^{-s}$, we have 
\begin{equation}
    \mathcal{L}(u,\chi_D)=\sum_{n=0}^{\text{deg}(D)-1}\sigma_n(D)u^n
\end{equation}
where $\sigma_n(D)=\sum_{f\in\mathbb{A}^+_n}\chi_D(f)$. Fix a generator $\gamma$ of $\mathbb{F}_q^*$ and write $\bar{D}=\gamma D$ for any $D\in\mathcal{H}_{2g+2}$. We have 
\begin{equation}\label{eq:2.2}
    \sigma_n(\bar{D})=(-1)^n\sigma_n(D).
\end{equation}
For $D\in\mathcal{H}_{2g+2}$, $\mathcal{L}(u,\chi_{\bar{D}})$ has a trivial zero at $u=-1$. Therefore we can define the complete L-function, $\mathcal{L}^*(u,\chi_{\bar{D}}),$ as
\begin{equation}\label{eq:2.3}
    \mathcal{L}^*(u,\chi_{\bar{D}})=(1+u)^{-1}\mathcal{L}(u,\chi_{\bar{D}}),
\end{equation}
which is a polynomial of degree $2g$.
\begin{lemma}
Let $\chi_{\gamma D}$ be a quadratic character, $\gamma$ a fixed generator of $\mathbb{F}_q^*$ and  $D\in\mathcal{H}_{2g+2}$. Then
\begin{equation}\label{eq:3.8}
    \mathcal{L}(q^{-2},\chi_{\bar{D}})=\sum_{f\in\mathbb{A}^+_{\leq 2g}} (-1)^{\text{deg}(f)}\frac{\chi_D(f)}{|f|^2}+q^{-4g-2}\sum_{f\in\mathbb{A}^+_{\leq 2g}}\chi_D(f)
\end{equation}
\end{lemma}
\begin{proof}
Write
\begin{equation}\label{eq:2.5}
    \mathcal{L}^*(u,\chi_{\bar{D}})=\sum_{n=0}^{2g}\sigma^*_n(\bar{D})u^n.
\end{equation}
Using (\ref{eq:2.3}), we have 
\begin{equation}\label{eq:2.6}
    \sigma^*_n(\bar{D})=\sum_{i=0}^n(-1)^{n-i}\sigma_i(\bar{D}).
\end{equation}
Therefore, using (\ref{eq:2.2}), (\ref{eq:2.5}) and (\ref{eq:2.6}) we have
\begin{equation*}
    \mathcal{L}^*(q^{-2},\chi_{\bar{D}})=\sum_{n=0}^{2g}\sigma^*(\bar{D})q^{-2n}=\sum_{n=0}^{2g}\sum_{i=0}^n(-1)^n\sigma_i(D)q^{-2n}.
    \end{equation*}
 The result follows.
 \end{proof}
\par\noindent
Following the same arguments as presented in \cite{Andrade2012}, section 4, we have
\begin{align}\label{eq:2.7}
    \sum_{D\in\mathcal{H}_{2g+2}}L(2,\chi_{\bar{D}}) &=\sum_{D\in\mathcal{H}_{2g+2}}\sum_{\substack{f\in\mathbb{A}^+_{\leq 2g}\\f=\square}}(-1)^{\text{deg}(f)}\frac{\chi_D(f)}{|f|^2}+\sum_{D\in\mathcal{H}_{2g+2}}\sum_{\substack{f\in\mathbb{A}^+_{\leq 2g
  }\\f\neq\square}}(-1)^{\text{deg}(f)}\frac{\chi_D(f)}{|f|^2}\nonumber\\
  &+q^{-2-4g}\sum_{D\in\mathcal{H}_{2g+2}}\sum_{\substack{f\in\mathbb{A}^+_{\leq 2g}\\f=\square}}\chi_D(f)+q^{-2-4g}\sum_{D\in\mathcal{H}_{2g+2}}\sum_{\substack{f\in\mathbb{A}^+_{\leq 2g}\\f\neq\square}}\chi_D(f).
\end{align}
\begin{lemma}
We have
\begin{equation*}
    \sum_{\substack{D\in\mathcal{H}_{2g+2}\\(D,l)=1}}1=\frac{|D|}{\zeta_A(2)}\prod_{P|l}\frac{|P|}{|P|+1}+O\left(|D|^{\frac{1}{2}}|l|^{\epsilon}\right).
\end{equation*}
\end{lemma}
\begin{proof}
See \cite{Andrade2012}, Proposition 5.2. 
\end{proof}
\begin{lemma}
We have 
\begin{equation*}
    \sum_{l\in\mathbb{A}^+_m}\prod_{P|l}\frac{|P|}{|P|+1}=q^m\sum_{d\in\mathbb{A}^+_{\leq m}}\frac{\mu(d)}{|d|}\prod_{P|d}\frac{1}{|P|+1}.
\end{equation*}
\end{lemma}
\begin{proof}
See \cite{Andrade2012}, Lemma 5.7.
\end{proof}

\begin{lemma}
We have
\begin{enumerate}
\item
\begin{equation*}
    \sum_{d\in\mathbb{A}^+_{\leq g}}\mu(d)\prod_{P|d}\frac{1}{|P|+1}\leq g+1.
\end{equation*}
    \item For $s=1$ or $4$, we have 
\begin{equation*}
    \sum_{d\in\mathbb{A}^+_{\leq g}}\frac{\mu(d)}{|d|^s}\prod_{P|d}\frac{1}{|P|+1}=P(s)+ O(q^{-sg}).
\end{equation*}
\end{enumerate}
\end{lemma}
\begin{proof}
See \cite{Jung2014AEnsemble}, Lemma 3.5 and Lemma 3.3. 
\end{proof}
\begin{lemma}
If $f\in\mathbb{A}$ is not a perfect square then
\begin{equation}
    \sum_{\substack{D\in\mathcal{H}_{2g+2}\\f\neq \square}}\left(\frac{D}{f}\right)\ll |D|^{\frac{1}{2}}|f|^{\frac{1}{4}}.
\end{equation}
\end{lemma}
\begin{proof}
See \cite{SimpleLaValeurMoyenneDe}, Lemma 4.3.
\end{proof}
\section{Proof of Main Theorem}
From Lemma 2.5, we can show that the terms corresponding to the contribution of non-square $f$ in \ref{eq:2.7} are bounded by $q^g$ and $q^{-\frac{g}{2}}$ respectively. Therefore, it remains to calculate the contribution of the square polynomials $f$ in (\ref{eq:2.7}), which are calculated using the following results.
\begin{prop}
We have
\begin{equation}
    \sum_{D\in\mathcal{H}_{2g+2}}\sum_{\substack{f\in\mathbb{A}^+_{\leq 2g}\\f=\square}}(-1)^{\text{deg}(f)}\frac{\chi_D(f)}{|f|^2}=\frac{|D|}{\zeta_{\mathbb{A}}(2)}\zeta_{\mathbb{A}}(4)P(4)+O(q^{g(1+\epsilon)}).
\end{equation}
\end{prop}
\begin{proof}
Using Lemma 2.2 we have 
\begin{equation*}
    \sum_{D\in\mathcal{H}_{2g+2}}\sum_{\substack{f\in\mathbb{A}^+_{\leq 2g}\\f=\square}}(-1)^{\text{deg}(f)}\frac{\chi_D(f)}{|f|^2}=\frac{|D|}{\zeta_{\mathbb{A}}(2)}\sum_{m=0}^{g}q^{-4m}\sum_{l\in\mathbb{A}^+_{m}}\prod_{P|l}\frac{|P|}{|P|+1}+O\left(|D|^{\frac{1}{2}}\sum_{n=0}^{2g}q^{n\epsilon-n}\right).
\end{equation*}
Invoking Lemma 2.3, we have 
\begin{align*}
    \sum_{D\in\mathcal{H}_{2g+2}}\sum_{\substack{f\in\mathbb{A}^+_{\leq 2g}\\f=\square}}(-1)^{\text{deg}(f)}\frac{\chi_D(f)}{|f|^2}&=\frac{|D|}{\zeta_{\mathbb{A}}(2)}\sum_{d\in\mathbb{A}^+_{\leq g}}\frac{\mu(d)}{|d|}\prod_{P|d}\frac{1}{|P|+1}\sum_{\text{deg}(d)\leq m\leq g}q^{-3m}
    +O\left(q^{-g}\frac{q^{2g\epsilon+\epsilon}-q^{2g+1}}{q^\epsilon-q}\right).
\end{align*}
Thus we get
\begin{align*}
    \sum_{D\in\mathcal{H}_{2g+2}}\sum_{\substack{f\in\mathbb{A}^+_{\leq 2g}\\f=\square}}(-1)^{\text{deg}(f)}\frac{\chi_D(f)}{|f|^2}&=\frac{|D|}{\zeta_{\mathbb{A}}(2)}\zeta_{\mathbb{A}}(4)\sum_{d\in\mathbb{A}^+_{\leq g}}\frac{\mu(d)}{|d|^4}\prod_{P|d}\frac{1}{|P|+1}\\&-\frac{q^{2-g}}{\zeta_{\mathbb{A}}(2)(q^3-1)}\sum_{d\in\mathbb{A}^+_{\leq g}}\frac{\mu(d)}{|d|}\prod_{P|d}\frac{1}{|P|+1}+O(q^{g(1+\epsilon)}).
\end{align*}
Using Lemma 2.4 parts 2 and 3 proves the Proposition. 
\end{proof}
\begin{lemma}
We have 
\begin{equation}
    q^{-2-4g}\sum_{D\in\mathcal{H}_{2g+2}}\sum_{\substack{f\in\mathbb{A}^+_{\leq 2g}}}\chi_D(f)\ll q^{-g}.
\end{equation}
\end{lemma}
\begin{proof}
Using Lemma 2.2, we have
\begin{equation*}
    q^{-2-4g}\sum_{D\in\mathcal{H}_{2g+2}}\sum_{\substack{f\in\mathbb{A}^+_{\leq 2g}\\f=\square}}\chi_D(f)\ll q^{-2g}\sum_{m=0}^g\sum_{l\in\mathbb{A}^+_{m}}\prod_{P|l}\frac{|P|}{|P|+1}.
\end{equation*}
Invoking Lemma 2.3, we get
\begin{align*}
    \sum_{D\in\mathcal{H}_{2g+2}}\sum_{\substack{f\in\mathbb{A}^+_{\leq 2g}\\f=\square}}\chi_D(f)&\ll q^{-2g}\sum_{d\in\mathbb{A}^+_{\leq g}}\frac{\mu(d)}{|d|}\prod_{P|d}\frac{1}{|P|+1}\sum_{\text{deg}(d)\leq m\leq g}q^m\\
    &\ll q^{-g}\sum_{d\in\mathbb{A}^+_{\leq g}}\frac{\mu(d)}{|d|}\prod_{P|d}\frac{1}{|P|+1}-q^{-2g}\sum_{d\in\mathbb{A}^+_{\leq g}}\mu(d)\prod_{P|d}\frac{1}{|P|+1}.
\end{align*}
Parts 1 and 2 from Lemma 2.4 prove the result.
\end{proof}
\par\noindent
Combining results from this section, we get that 
\begin{equation}
    \sum_{D\in\mathcal{H}_{2g+2}}L(2,\chi_{\gamma D})=\frac{q^{2g+2}}{\zeta_{\mathbb{A}}(2)}\zeta_{\mathbb{A}}(4)P(4)+O(q^{g(1+\epsilon)}).
\end{equation}
Using Proposition 2.3 in \cite{Rosen2002}, together with Proposition 1.1 b), completes the proof of Theorem 1.4.\\
\par\noindent
\textbf{Acknowledgement:} The author is grateful to the Leverhulme Trust (RPG-2017-320) for the for the support given during this research through a PhD studentship. The author would also like to thank Dr. Julio Andrade for suggesting this problem to me, as well as his useful advice during the course of the research. 
\bibliographystyle{plain}
\bibliography{K2Ogroup}
Department of Mathematics, University of Exeter, Exeter, EX4 4QF, UK\\
\textit{E-mail Address:}jm1015@exeter.ac.uk
\end{document}